\documentclass[12pt]{amsart}

\usepackage{xcolor}
\usepackage{graphicx}
\usepackage{amsmath}
\usepackage{amsfonts}
\usepackage{amssymb}
\usepackage[all]{xy}
\usepackage{array}
\usepackage{enumerate}
\usepackage{url}
\usepackage{tikz}
\usetikzlibrary{matrix,arrows}

\newtheorem{theorem}{Theorem}[section]

\newtheorem{conjecture}{Conjecture}
\newtheorem{corollary}[theorem]{Corollary}

\newtheorem{lemma}[theorem]{Lemma}

\newtheorem{problem}{Problem}
\newtheorem{proposition}[theorem]{Proposition}

\newcommand{\DD}{\mathcal{D}}

\title[On Classes of Antilattices]{On Elementary, Odd, Semimagic and Other Classes of Antilattices}
\author{Karin Cvetko-Vah \and
Michael Kinyon \and
Toma\v{z} Pisanski
}
\address[Cvetko-Vah]{Department of Mathematics \\
Faculty of Mathematics and Physics \\
University of Ljubljana \\ Jadranska 21, SI-1000 Ljubljana, Slovenia}
\email{karin.cvetko@fmf.uni-lj.si}

\address[Kinyon]{Department of Mathematics \\
University of Denver \\ Denver, CO 80208, USA}
\email{mkinyon@du.edu}

\address[Pisanski]{Faculty of Mathematics, Natural Sciences and Information Technologies \\
University of Primorska \\ Glagolja\v{s}ka 8 \\ SI-6000 Koper, Slovenia \\
and \\
Andrej Maru\v{s}i\v{c} Institute \\
University of Primorska \\ Muzejski trg 2 \\ 6000 Koper, Slovenia\\
IMFM \\ Jadranska 19, SI-1000 Ljubljana, Slovenia}
\email{pisanski@upr.si}

\date{\today}



\begin{document}
\begin{abstract}
An \emph{antilattice} is an algebraic structure based on the same set of axioms as a lattice except that the two commutativity axioms for $\land$ and $\lor$ are replaced by anticommutative counterparts. In this paper we study certain classes of antilattices, including elementary (no nontrivial subantilattices), odd (no subantilattices of order $2$), simple (no nontrivial congruences) and irreducible (not expressible as a direct product). In the finite case, odd antilattices are the same as Leech's \emph{Latin} antilattices which arise from the construction of semimagic squares from pairs of orthogonal Latin squares.
\end{abstract}
\maketitle

\section{Introduction}

\emph{Antilattices} are algebraic structures introduced and studied by J. Leech in 2005 \cite{Le05} and later expanded in the book \cite{Le20}. Antilattices have two binary operations $\land$, $\lor$, and fulfil almost the same set of axioms as lattices, except that the commutativity axioms are replaced by axioms of anticommutativity, in the sense that no two distinct elements commute with respect to either operation. In \cite{LaLe02}  it was shown that antilattices play a structural role in the theory of noncommutative lattices, where ``noncommutative'' is to be understood as ``not necessarily commutative''. For introductions to the modern theory of noncommutative lattices, see \cite{Le19,Le20}.

In both \cite{Le05} and \cite{Le20}, Leech explored the connection between certain types of antilattices, which we here call \emph{Latin}, with (semi)magic squares. In this paper, we carry out this approach a bit further. Among other things we show that finite antilattices without proper subantilattices of order $2$ can be put into one-to-one correspondence with pairs of orthogonal Latin squares.

In {\S}\ref{Sec:preliminaries}, we give some background information. We briefly review a few necessary notions from universal algebra; the reader familiar with this can safely skip that subsection. We then review some basic facts about bands, i.e., idempotent semigroups, especially \emph{rectangular} bands. In {\S}\ref{Ssec:genmat} we discuss in detail generating matrices of rectangular bands; these are essentially a combinatorial expression of the algebraic fact that rectangular bands are direct products of left zero and right zero bands. In {\S}\ref{Sec:testing}, we discuss a linear  time algorithm for testing if a band is rectangular using a generating matrix.

In {\S}\ref{Sec:antilattices} we define antilattices in a broader context of noncommutative lattice theory. In {\S}\ref{Sec:congruences}, we discuss congruences of antilattices and how they correspond to a particular type of partition of generating matrices. We also discuss simple and irreducible antilattices.

In {\S}\ref{Sec:semi} we discuss \emph{semimagic} and the aforementioned Latin antilattices. The latter are constructed via the Choi-Euler construction of semimagic squares from pairs of orthogonal Latin squares and we discuss this in some detail.

Returning to the algebraic side of the subject, in {\S}\ref{Sec:elemodd} we introduce \emph{elementary} antilattices (those with no nontrivial subantilattices) and \emph{odd} antilattices (those with no subantilattices of order $2$). In the finite case, odd antilattices turn out be precisely the same as Latin antilattices (Theorem \ref{Thm:odd_latin}). Oddness is also equivalent to the emptiness of the \emph{even graph} of an antilattice. We also connect to the paper \cite{CvKiLePi19} by showing that a \emph{regular} odd antilattice must be trivial (Theorem \ref{Thm:reg_odd}).

Finally, {\S}\ref{Sec:main} is devoted to our main result Theorem \ref{Thm:main}, which describes all the implications between the various classes of antilattices considered in this paper. We especially discuss examples based on orthogonal Latin squares which show which implications are not reversible.

\section{Preliminaries}
\label{Sec:preliminaries}

\subsection{Universal algebra}
\label{Ssec:ua}
Our approach to antilattices is a mix of both combinatorial and universal algebra methods, and so for the benefit of the
reader unfamiliar with the latter, we briefly review the needed notions. A standard reference is \cite{Bu81}.

An algebraic structure, such as a group or a lattice, is a set together with a collection of operations called its signature.
A class of algebraic structures with the same signature is called a \emph{variety} if it is axiomatized by a set of (universally
quantified) identities. For example, lattices form a variety, defined by the above identities.

For algebras with the same signature one may define homomorphisms, subalgebras and direct products. Birkhoff's fundamental
HSP (Homomorphism, Subalgebra, Product) Theorem \cite{Bi35,Bu81} states that an algebraic structure defines a
variety if and only if it is closed under homomorphisms, subalgebras and direct products.

A \emph{quasivariety} is a class of algebraic structures axiomatized by \emph{quasi-identities}, which are either identities or formulas or the form $(s_1 = t_1\ \&\ \ldots\ \&\ s_k = t_k)\implies s = t$ where the $s,t,s_i,t_i$ are terms (formulas formed from just variables and operations). A quasivariety is \emph{proper} if it is not a variety. For example, cancellative semigroups form a proper quasivariety defined
by the associative law and the cancellation quasi-identities $xy=xz\implies y=z$ and $yx=zx\implies y=z$.

The analog of Birkhoff's HSP Theorem for quasivarieties is Mal'cev's Theorem \cite{Mal}, which states that an algebraic structure is a quasivariety if and only if it is closed under subalgebras, direct products and ultraproducts. Proper quasivarieties are not closed under taking homomorphic images.

A \emph{congruence} $\alpha$ on an algebraic structure $A$ in a quasivariety is an equivalence relation which, as a set of ordered pairs, is also a subalgebra of $A\times A$. If $A$ lies in a variety, then the quotient $A/\alpha$ lies in the same variety. The First Isomorphism Theorem holds in this setting: congruences are precisely kernel relations of homomorphisms, and a homomorphism's image is isomorphic to the quotient by its kernel.

An algebraic structure $A$ is (congruence) \emph{simple} if it has only two congruences, the diagonal (or identity) congruence $\nabla$ and the universal congruence $\Delta$. (Note that in some areas of algebra such as semigroup theory, the word ``simple'' is used in a different sense and the concept we describe here is instead called \emph{congruence-free} (\cite{How95}, p.~93).

An algebraic structure $A$ is said to be \emph{irreducible} if it is not isomorphic to a direct product $B\times C$ of nontrivial algebras.
Here the meaning of ``nontrivial'' is relative to the class of algebraic structures under consideration. In the context of this paper, nontrivial will mean having cardinality greater than $1$.

\subsection{Bands, semilattices and rectangular bands}
\label{Ssec:bands}
A \emph{magma} $(B,\cdot)$ is a set $B$ with a binary operation $\cdot$. As is customary when only one binary operation is present, we abbreviate it by juxtaposition: $xy=x\cdot y$. A magma is a \emph{semigroup} if it satisfies the associative identity $(xy)z=x(yz)$ for all $x,y,z$, and a magma is \emph{idempotent} if it satisfies $xx=x$ for all $x$. A \emph{band} is an idempotent semigroup. Thus bands form a variety. Subvarieties of the variety of all bands have been classified \cite{GePe1989}.

Every band has a \emph{natural preorder} $\preceq$ defined by $x\preceq y\,\iff\,x = xyx$, and a \emph{natural partial order} $\leq$ defined by $x\leq y\,\iff\,x = xy = yx$. The latter relation refines the former in the sense that $\leq\ \subseteq\ \preceq$.

The equivalence relation associated to the natural preorder is known as \emph{Green's} $\mathcal{D}$-\emph{relation}, defined by $x\,\DD\,y\,\iff\,x\preceq y\ \&\ y\preceq x$. This relation is a congruence. \emph{Green's} $\mathcal{L}$- and $\mathcal{R}$-\emph{relations} are defined, respectively, by $x\,\mathcal{L}\,y\ \iff ( xy=x\text{ and }yx=y)$, and $x\,\mathcal{R}\,y\ \iff (xy=y\text{ and }yx=x)$.

A \emph{semilattice} is a commutative band, that is, a band satisfying the additional identity $xy = yx$ for all $x,y$. Semilattices are precisely those bands in which $\DD$ is the identity relation.

A \emph{rectangular band} is a band satisfying the \emph{anticommutativity} quasi-identity $xy = yx \implies x = y$.
Thus rectangular bands form a quasivariety, but one can say more. It turns out that rectangular bands form a variety characterized by the identities of associativity, idempotence and the identity $xyx = x$ for all $x,y$. It follows that rectangular bands are precisely those bands in which $\DD$ is the universal relation.

The main structural result about general bands is the \emph{Clifford-McLean Theorem}: Every band is a semilattice of rectangular bands. More precisely, if $B$ is a band, then $B/\DD$ is a semilattice. Thus we can visualize a band as a Hasse diagram for a semilattice, where each node is a $\DD$-class.

It follows that a simple band is either a semilattice or a rectangular band. The only simple semilattices are the $1$-element semilattice and the $2$-element semilattice. We will see below what the simple rectangular bands are.

A \emph{left zero band} is a band satisfying $xy=x$ for all $x,y$, and similarly, a \emph{right zero band} is a band satisfying $xy=y$ for all $x,y$. Left zero bands and right zero bands are rectangular. Any permutation of a left [right] zero band is an automorphism of the band. It follows that for each positive integer $n$, there is only one left [right] zero band of size $n$ up to isomorphism.

Every rectangular band $B$ is isomorphic to a direct product of a left zero band and a right zero band (\cite{Le20}, Thm.~1.2.4). Indeed, fix $a\in B$, observe that $Ba = \{ xa\mid x\in B\}$ is a left zero band, $aB = \{ ay\mid y\in B\}$ is a right zero band, and then check that the mapping $Ba\times aB\to B; (xa,ax)\mapsto x$ is an isomorphism.

It follows that an irreducible rectangular band is a left zero band or a right zero band. Further, a left [right] zero band of composite order $mn$ is isomorphic to the direct product of a left [right] zero band of order $m$ and a left [right] zero band of order $n$. We conclude that a rectangular band is irreducible if and only if it is a $1$-element band or a left zero band of prime order or a right zero band of prime order.

\begin{lemma}\label{Lem:simple_irr}
Every simple band is irreducible.
\end{lemma}
\begin{proof}
If $B$ and $B'$ are nontrivial bands, then the product band $N\times N'$ has two nontrivial congruences which are given by the kernels of the projection homomorphisms $B\times B'\to B$ and $B\times B'\to B$. Thus a simple band must be irreducible.
\end{proof}

A nontrivial simple rectangular band, being irreducible, is a left zero band or a right zero band of prime order. Any equivalence relation on a rectangular band is a congruence, hence any partition of a rectangular band is the partition of a congruence. Thus if the order of rectangular band is greater than $2$, then the band will have a nontrivial congruence. It follows that the simple rectangular bands are the $1$-element band, the $2$-element left zero band, and the $2$-element right zero band.

Putting this together with our earlier discussion, we conclude that there are precisely $4$ simple bands: the $1$-element band, the $2$-element semilattice, the $2$-element left zero band, and the $2$-element right zero band.

We conclude this subsection with a useful elementary observation.

\begin{lemma}\label{Lem:subband}
  Let $B$ be a band and let $\alpha$ be a congruence on $B$. Then each $\alpha$-class is a subband.
\end{lemma}
\begin{proof}
  This follows from idempotence: if $a\,\alpha\,b$ then $ab\,\alpha\,aa=a$.
\end{proof}

%

\subsection{Generating matrices}
\label{Ssec:genmat}
The preceding considerations lead to a different representation of rectangular bands. A rectangular band is determined by a rectangular array \cite{Le20} called a \emph{generating matrix}. In particular we have the following result.

\begin{proposition}[\cite{Le20}, p.~21]
A band $(B,\cdot)$ of order $n$ is rectangular if and only if there exist $p,q$ such that $n = pq$, and a $p\times q$ matrix $G$ with distinct entries from $B$, called a \emph{generating matrix} for $B$, such that
\[
(G)_{ij}\cdot (G)_{kl} = (G)_{il}
\]
for all $i,k\in \{1,\ldots,p\}$, $j,l\in \{1,\ldots,q\}$.
\end{proposition}

\[
\xymatrix{
\ar@{=}[ddddd]\ar@{=}[rrrrrr]&&&&&&\ar@{=}[ddddd]\\
&x\ar@{-}[l]\ar@{-}[u]\ar@{-}[rrrr]\ar@{-}[ddd]&&&&x \bullet y\ar@{-}[r]\ar@{-}[u]\ar@{-}[ddd] &\\
&&&&&&\\
&&&&&&\\
&y \bullet x\ar@{-}[l]\ar@{-}[d]\ar@{-}[rrrr]&&&&y\ar@{-}[r]\ar@{-}[d]&\\
\ar@{=}[rrrrrr]&&&&&&\\
   }
\]

The rows of a generating matrix consist of the $\mathcal{R}$-classes and the columns consist of the $\mathcal{L}$-classes. Since
any simultaneous permutation of rows and columns will preserve the relations of being in the same row or same column, such a
permutation will transform the generating matrix into another one for the same rectangular band.

If a generating matrix for a band has size $n=pq$, then the pair $(p,q)$ is an invariant called its \emph{type}. If $p=q$, the rectangular band is called \emph{square}. If $p=1$ or $q=1$, the band is called \emph{flat}. Any flat rectangular band is either a left zero band or a right zero band. Any rectangular band of prime order is flat.

A $p\times q$ generating matrix for a rectangular band on the elements $1,\ldots,pq$ is said to be in \emph{normal form} if, for each $i\in \{1,\ldots,p\}$, the $i$th row consists of the elements $(i-1)q+1,\ldots,iq$ in order. Every finite rectangular band is isomorphic to one with generating matrix in normal form, just by relabeling the band elements. It follows that finite rectangular bands are determined up to isomorphism by their types.

Let $B$ be a rectangular band, and let $I$ and $J$ be, respectively, a left zero band and a right zero band such that $B\cong I\times J$.
If $\alpha$ is a congruence of $B$, then there exist congruences $\alpha_{\ell}$ of $I$ and $\alpha_r$ of $J$ such that, for $i_1, i_2\in I$, $j_1, j_2\in J$, $(i_1,j_1)\,\alpha\,(i_2,j_2)$ if and only if $i_1\,\alpha_{\ell}\,i_2$ and $j_1\,\alpha_r\,j_2$ (\cite{How95}, p.114).

If we interpret congruences in terms of their associated partition, then congruences in rectangular bands have a simple description using generating matrices \cite{Le05}. In particular, any congruence on a rectangular band is characterized by a \emph{cartesian partition} of its generating matrix, that is, a partition of the matrix induced by a partition of its rows and a partition of its columns. In this case, each equivalence class of the partition of the band elements is, in fact, a congruence class. By permuting rows and columns, the matrix can be arranged so that elements in the same block of the cartesian partition lie in a submatrix.

For example, here is a cartesian partition of a generating matrix for a rectangular band of size $15$, along with the corresponding quotient generating matrix, where the entries of the latter are labeled by the entries in the upper left corner of each submatrix of the former.
\[
\begin{array}{|ccc|cc|}
\hline
1 & 2 & 3 & 4 & 5 \\
6 & 7 & 8 & 9 & 10 \\
\hline
11 & 12 & 13 & 14 & 15 \\
\hline
\end{array}
\qquad
\begin{array}{|cc|}
\hline
\overline{1} & \overline{4} \\
\overline{11} & \overline{14} \\
\hline
\end{array}
\]

For finite rectangular bands, the conclusion of Lemma \ref{Lem:subband} can be seen combinatorially: a congruence on a band corresponds to a cartesian partition of a generating matrix, and each submatrix in the partition is itself a generating matrix, hence corresponds to a subband.

Finally, let $B$ and $C$ be rectangular bands of types $(p,q)$ and $(r,s)$, respectively, and with generating matrices $G$ and $H$, respectively. Then the direct product $B\times C$ is of type $(pr,qs)$ and has generating matrix $G\otimes H$ where $\otimes$ denotes the Kronecker or tensor product of matrices.

\section{Testing rectangularity}
\label{Sec:testing}

Testing whether a magma on an $n$-element set satisfies a particular identity typically involves looking up the values of the operation in its $n\times n$ Cayley table. For example, testing if a magma is idempotent can be done in linear time $O(n)$ just by checking the $n$ values $xx$ for each $x$. On the other hand, the best known deterministic algorithms, such as Light's Associativity Test, for testing if a magma is a semigroup have a worst-case runtime of $O(n^3)$ steps. Thus testing if a magma is a band takes at worst $O(n^3)$ steps.

If a magma is already known to be a band, then testing if it is rectangular using its Cayley table takes $O(n^2)$ steps to verify the identity $xyx=x$ or equivalently, the quasi-identity $xy=yx\implies x=y$. However, by using a generating matrix, testing for rectangularity can be done in linear time $O(n)$. This is shown by the following four-step algorithm.

\begin{enumerate}
\item Choose an arbitrary element $a$ and place it in the top left corner of the matrix.
\item Fill in the first row with all elements of the form $ay$.
\item Fill in the first column with all elements of the form $xa$.
\item For each $x$ in the first column and for each $y$ in the first row, compute $u = xy$ and $v = yx$. Place $u$ in the intersection of the row of $x$ and the column of $y$ and verify that $v=a$.
\end{enumerate}

Step 1 can be performed in constant time $O(1)$. Steps 2, 3 and 4 require linear time $O(n)$. Thus the total running time of the algorithm is linear.

Let $q$ be the length of the first row and let $p$ be the length of the first column.
The algorithm detects if a band is not rectangular if, for any reason, the matrix fails to fill up with all the band entries,
for instance, if $pq < n$ or if the same element occurs in two different entries. If $pq=n$ and the matrix fills up with all band entries, then the band is rectangular. To see this, note that if there were distinct elements $x$ and $y$ such that $xy=yx$, then there would exist two locations in the matrix with the same entry, a contradiction.

What remains is to show that the matrix we have constructed is a generating matrix, that is, for distinct elements $x$ and $y$, $xy$ should be in the same row as $x$ and the same column as $y$. To see this, notet that the elements $xa$ and $ya$ are in the first column and in the same row as, respectively, $x$ and $y$. Similarly, the elements $ax$ and $ay$ are in the first row and in the same column as, respectively, $x$ and $y$. Thus $xy = [xa\cdot ax][ya\cdot ay] = xa\cdot axya\cdot ay = xa\cdot a\cdot ay = xa\cdot ay$.

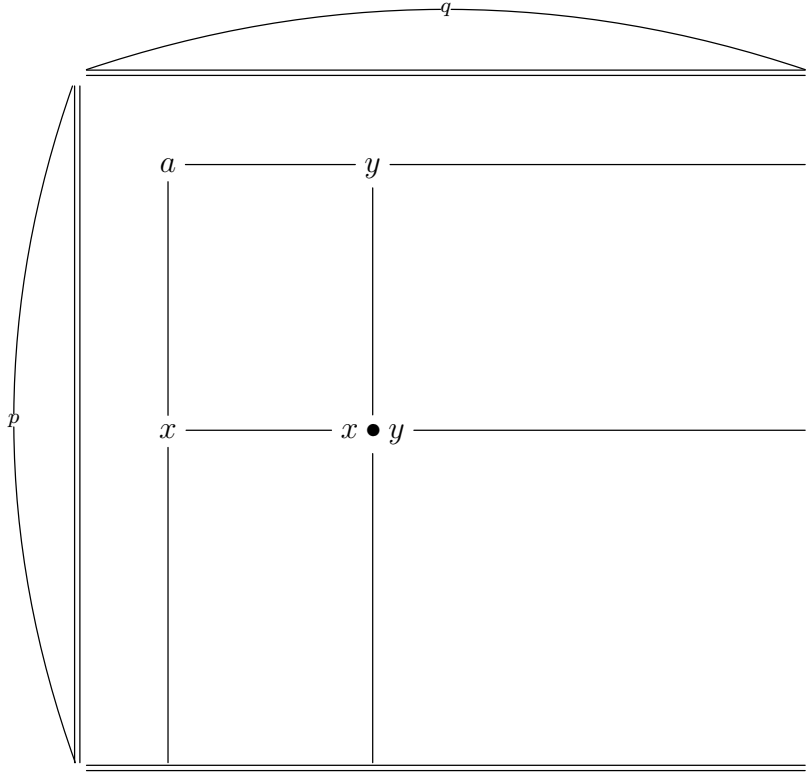
\begin{figure}[th]
\begin{center}
\[
\xymatrix{
\ar@{=}[dddddddd]\ar@{-}@/_2pc/[dddddddd]|{p}\ar@{=}[rrrrrrrr]\ar@{-}@/^2pc/[rrrrrrrr]|{q}&&&&&&&&\ar@{=}[dddddddd]\\
&a \ar@{-}[rr]\ar@{-}[ddd]  && y\ar@{-}[rrrrr]\ar@{-}[ddd]&&&&& \\
&&&&&&&&&&&&&&\\
&&&&&&&&&&&&&&\\
&x\ar@{-}[rr]\ar@{-}[dddd]  && x \bullet y\ar@{-}[rrrrr]\ar@{-}[dddd]&&&&& \\
&&&&&&&&&&&&&&\\
&&&&&&&&&&&&&&\\
&&&&&&&&&&&&&&\\
\ar@{=}[rrrrrrrr]&&&&&&&&\\
   }
\]
\caption{Construction of the generating matrix of type $(p,q)$ in linear time. Note that $a = y \bullet x$}
\end{center}
\end{figure}

Finally, we note that this algorithm only works if we know \emph{a priori} that the given magma is a band. It does not allow us to bypass associativity testing. Indeed, consider the idempotent magma given by the following Cayley table:
\[
\begin{array}{c|cccc}
\cdot & 1 & 2 & 3 & 4 \\ \hline
    1 & 1 & 2 & 1 & 1 \\
    2 & 1 & 2 & 3 & 1 \\
    3 & 3 & 4 & 3 & 1 \\
    4 & 1 & 1 & 1 & 4
\end{array}
\]
This magma is not a semigroup because $(1\cdot 2)\cdot 3 = 2\cdot 3 = 3$, but $1\cdot (2\cdot 3) = 1\cdot 3 = 1$.
Nevertheless, if we follow the algorithm starting with the element $1$, we fill in a $2\times 2$ ``generating matrix''
\[
\begin{array}{|cc|}
\hline
    1 & 2 \\
    3 & 4 \\
\hline
\end{array}
\]

\section{Antilattices}
\label{Sec:antilattices}
\subsection{Definitions}
A \emph{double band} $(N,\land,\lor)$ is a set $N$ together with two associative and idempotent operations $\land$, $\lor$.
In particular, the \emph{reducts} $(N,\lor)$ and $(N,\land)$ are bands.

A \emph{quasilattice} $(N,\land,\lor)$ is a double band satisfying the following pair of absorption laws.
\begin{align}
x\land (y\lor  x\lor  y)\land x &= x\,, \label{quasi_and_abs}\\
x\lor  (y\land x\land y)\lor  x &= x\,. \label{quasi_or_abs}
\end{align}
These identities express the duality $x\land y\land x = x$ if and only if $y\lor x\lor y = y$, that is, $x\preceq y$ under $\land$ if and only if $y\preceq x$ under $\lor$.

An \emph{antilattice} $(N,\land,\lor)$ is a double band such that the reducts $(N,\land)$ and $(N,\lor)$ are rectangular bands. Thus besides associativity and idempotence, antilattices satisfy the anticommutativity identities:
\begin{align}
x \land y \land x = x \label{anti_and}\\
x \lor y \lor x = x  \label{anti_or}
\end{align}

It is evident that every antilattice is a quasilattice. Note that since $\vee$ does not occur in \eqref{anti_and} and $\lor$ does not occur in \eqref{anti_or}, any two rectangular band structures on the same underlying set determine an antilattice.

\subsection{Generating matrices and types}
Any antilattice $N$ is determined by a pair of generating matrices, say $M$ (meet) of order $(p,q)$ and $J$ (join) of order $(r,s)$ where $pq = rs = n$. The quadruple $(p,q,r,s)$ is an antilattice invariant called the \emph{type} of $N$. Since $pq=rs$, the two matrices have the same number of entries. Any pair of matrices with the same number of entries is called \emph{compatible}.

If both reduct bands of an antilattice are square, i.e. $p = q = r = s$, or equivalently, if both generating matrices are square, the antilattice itself is called \emph{square}. If one of the reducts is square, the antilattice is called \emph{semisquare}. Square antilattices exist only for square orders, however not every antilattice of square order is square.

If both reduct bands of an antilattice are flat, the antilattice itself is called \emph{flat}; if only one of them is flat, the antilattice is called \emph{semiflat}.  Any antilattice of prime order is flat.

It will often be convenient if one of the generating matrices of an antilattice is in normal form. By convention, we will choose the meet matrix $M$ for this purpose. This can always be achieved by suitable relabelling of the antilattice elements \cite{CvKiLePi19}.

\subsection{Example}
Let $N=\{0,1,2,3\}$ and let the operations be given by the following tables:

\[
\begin{tabular}{r|rrrr}
$\land$ & 0 & 1 & 2 & 3\\
\hline
    0 & 0 & 1 & 2 & 3 \\
    1 & 0 & 1 & 2 & 3 \\
    2 & 0 & 1 & 2 & 3 \\
    3 & 0 & 1 & 2 & 3
\end{tabular}
\qquad
\begin{tabular}{r|rrrr}
$\lor $& 0 & 1 & 2 & 3\\
\hline
    0 & 0 & 2 & 2 & 0 \\
    1 & 3 & 1 & 1 & 3 \\
    2 & 0 & 2 & 2 & 0 \\
    3 & 3 & 1 & 1 & 3
\end{tabular}
\]
$N$ may given by generating matrices:
\begin{center}
$M = $ \begin{tabular}{|cccc|}
\hline
0 & 1 & 2 & 3 \\
\hline
\end{tabular}
\qquad
\medskip
$J = $ \begin{tabular}{|cc|}
\hline
0 & 2 \\
3 & 1 \\
\hline
\end{tabular}
\end{center}

\noindent $N$ is therefore of type $(1,4,2,2)$, and is both semiflat and semisquare.

\section{Congruences, quotients and products of antilattices, simple and irreducible antilattices}
\label{Sec:congruences}

Any pair of generating matrices with entries from the same set $N$ defines an antilattice on $N$. However, understanding subantilattices, quotients and products of antilattices is more complicated than in the rectangular band case because both operations must be considered.

\subsection{Congruences and simplicity}
Since congruences in rectangular bands are described by cartesian partitions of their generating matrices, the same is true for antilattices. For an antilattice $N$ with generating matrices $M$ and $J$, cartesian partitions of $M$ and $J$ are \emph{compatible} if they induce the same partition of $N$. Thus every congruence on $N$ can be described by a pair of compatible cartesian partitions.

Since both reducts of an antilattice are bands, Lemma \ref{Lem:subband} has the following immediate corollary.

\begin{corollary}\label{Cor:cong_sub}
Let $N$ be an antilattice and let $\alpha$ be a congruence of $N$. Then every $\alpha$-congruence class is a subantilattice.
\end{corollary}

Recall that a quasilattice is \emph{simple} if its only congruences are the diagonal congruence $\nabla$ and the universal congruence $\Delta$.

\begin{proposition}[\cite{Le20}]
A simple quasilattice is either a lattice or an antilattice.
\end{proposition}

There are only two simple lattices, namely the trivial lattice and the $2$-element lattice. Hence understanding simple quasilattices ``reduces'' to studying simple antilattices.

However, unlike the situation for rectangular bands described earlier, there is no easy classification of simple antilattices. It is known that there are no simple antilattices of odd prime order (\cite{Le20}, Lem.~3.2.1) or of order $4$ (\cite{Le20}, Prop.~3.2.2). On the other hand, simple antilattices exist for all composite orders greater than $5$ (\cite{Le20}, Thm.~3.2.3). The smallest example given by Leech's construction has $M$ a $2\times 3$ matrix in normal form and
\[
J = \begin{array}{|ccc|}
\hline
1 & 2 & 4 \\
5 & 6 & 3 \\
\hline
\end{array}
\]
A classification of finite simple antilattices seems way out of reach.

\subsection{Products of antilattices}
Let $N$ and $N'$ be antilattices with corresponding generating matrix pairs $(M,J)$ and $(M',J')$ and of types
$(p,q,r,s)$ and $(p',q',r',s')$. Then the product antilattice $N\times N'$ is of type $(pp',qq',rr',ss')$ and has generating
matrix pair $(M\otimes M', J\otimes J')$ where again, $\otimes$ denotes the Kronecker product of matrices.

Unlike the situation for rectangular bands, determining which antilattices are irreducible is not a straightforward task. For instance, it is easy to see that an antilattice of prime order is irreducible, but there exist irreducible antilattices of composite orders.

\begin{proposition}\label{Prp:simple_irred}
Each simple antilattice is irreducible, but there exist irreducible antilattices that are not simple.
\end{proposition}
\begin{proof}
The first assertion follows by the same argument as in the proof of Lemma \ref{Lem:simple_irr}. For the second assertion, consider the antilattice of order $6$ given by the following generating matrices:
\[
\begin{array}{|c|cc|}
\hline
1 & 2 & 3  \\
4 & 5 & 6   \\
\hline
\end{array}
\qquad
\begin{array}{|c|cc|}
\hline
1 & 2 & 5  \\
4 & 6 & 3   \\
\hline
\end{array}\,.
\]
This antilattice is clearly not a product of antilattices of order $2$ and $3$, and hence, is irreducible. However, the partition $\{\{1,4\},\{2,3,5,6\}\}$ is cartesian, and therefore the antilattice is not simple.
\end{proof}

\section{Semimagic and Latin antilattices}
\label{Sec:semi}

\subsection{Semimagic antilattices}
A (classical) \emph{semimagic square} is an $n\times n$ array consisting of distinct entries from $\{1,\ldots,n^2\}$ such that the sums of the numbers in each row and in each column are equal. A semimagic square is \emph{magic} if the two diagonal also sum to that same value. Any simultaneous permutation of rows and columns preserve the property of being semimagic.

An antilattice is said to be \emph{semimagic} if it has a generating matrix pair $(M,J)$ where $M$ is in normal form and $J$ is a semimagic square. A semimagic antilattice is \emph{magic} if $J$ is a magic square.

Semimagic and magic antilattices are not necessarily simple.
As an example, consider the \emph{D\"{u}rer antilattice}, where $J$ is the magic square from D\"{u}rer's \emph{Melancholia I}.
\smallskip

\centerline{\includegraphics[width=0.33\textwidth]{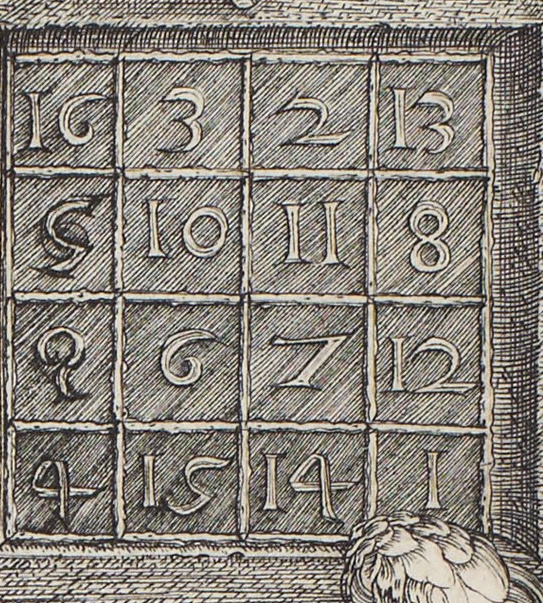}}

\noindent D\"{u}rer's antilattice is given by the matrices:

\begin{center}
\begin{tabular}{|cccc|}
\hline
1& 2 & 3 & 4  \\
5& 6 & 7 & 8  \\
9& 10 & 11 & 12  \\
13& 14 & 15 & 16  \\
\hline
\end{tabular}
\qquad
\begin{tabular}{|cccc|}
\hline
16& 3 & 2 & 13  \\
5& 10 & 11 & 8  \\
9& 6& 7& 12  \\
4& 15 & 14 & 1  \\
\hline
\end{tabular}
\end{center}

Some of the congruences of the D\"{u}rer antilattice were found by Leech \cite{Le05}. Here they are in
terms of the corresponding partitions:
\[
\begin{array}{ll}
\alpha_1: & [1,2,3,4,13,14,15,16|5,6,7,8,9,10,11,12]\\
\alpha_2: & [2,4,5,8,9,12,13,16|2,3,6,7,10,11,14,15]\\
\alpha_{12}: & [1,4,13,16|2,3,14,15|5,8,9,12|6,7,10,11]\\
\beta_1: & [1,2,13,14|3,4,15,16|5,6,9,10|7,8,11,12]\\
\beta_2: & [1,3,13,15|2,4,14,16|5,7,9,11|6,8,10,12]\\
\gamma_1:& [1,4,5,8|2,3,6,7|9,12,13,16|10,11,14,15]\\
\gamma_2: & [1,4,9,12|2,3,10,11|5,8,13,16|6,7,14,15]\\
\delta_1: & [1,13|2,14|3,15|4,16|5,9|6,10|7,11|8,12]\\
\delta_2: &[1,4|2,3|5,8|6,7|9,12|10,11|13,16|14,15]\\
\epsilon_{11}: & [1,13|2,14|3,15|4,16|5|9|6|10|7|11|8|12]\\
\epsilon_{12}: & [1|13|2|14|3|15|4|16|5,9|6,10|7,11|8,12]\\
\epsilon_{21}: & [1,4|2|3|5,8|6|7|9,12|10|11|13,16|14|15]\\
\epsilon_{22}: & [1|4|2,3|5|8|6,7|9|12|10,11|13|16|14,15]\,.
\end{array}
\]
Using a SageMath \cite{Sage} program, we found the following additional partitions:
\[
\begin{array}{ll}
\psi_1: & [1,4,13,16|2,3,14,15|5,8|6,7|9,12|10,11]\\
\psi_2: & [1,4,13,16|2,14|3,15|5,8,9,12|6,10|7,11]\\
\psi_3: &[1,13|2,3,14,15|4,16|5,9|6,7,10,11|8,12]\\
\psi_4: & [1,4|2,3|5,8,9,12|6,7,10,11|13,16|14,15]\\
\phi_1: & [1,4,13,16|2,14|3,15|5,8|6|7|9,12|10|11]\\
\phi_2: & [1|2,3|4|5,9|6,7,10,11|8,12|13|14,15|16]\\
\phi_3: & [1,13|2,3,14,15|4,16|5|6,7|8|9|10,11|12]\\
\phi_4: & [1,4|2|3|5,8,9,12|6,10|7,11|13,16|14|15]
\end{array}
\]

The Hasse diagram of the lattice of congruences for D\"{u}rer's antilattice is depicted below.

\begin{tiny}
\[
{\xymatrix{
             &   & \Delta  \ar@{-}[dl] \ar@{-}[dr]                  &           &                       \\
             & \alpha_1 \ar@{-}[dl]\ar@{-}[d]\ar@{-}[dr] &    & \alpha_2  \ar@{-}[dl]\ar@{-}[d]\ar@{-}[dr]&   \\
    \beta_1          \ar@{-}[ddrrr]            & \beta_2 \ar@{-}[ddrr]&  \alpha_{12} \ar@{-}[dll]\ar@{-}[dl]\ar@{-}[dr]\ar@{-}[drr] & \gamma_1  \ar@{-}[ddll]& \gamma_2 \ar@{-}[ddlll]  \\
    \psi_2     \ar@{-}[dd]\ar@{-}[drrr]\ar@{-}[ddrrr]                 & \psi_1 \ar@{-}[ddl]\ar@{-}[ddrrr]\ar@{-}[d]&    & \psi_3 \ar@{-}[d]\ar@{-}[ddll]\ar@{-}[ddr] & \psi_4  \ar@{-}[dlll]\ar@{-}[ddlll]\ar@{-}[ddl] \\
                        & \delta_2\ar@{-}[ddrr]\ar@{-}[ddrrr] &   & \delta_1 \ar@{-}[ddll]\ar@{-}[ddlll] &    \\
    \phi_1   \ar@{-}[d]\ar@{-}[drrr]                   & \phi_2 \ar@{-}[d]\ar@{-}[drrr]&    & \phi_4 \ar@{-}[dll]\ar@{-}[d]& \phi_3 \ar@{-}[d]\ar@{-}[dllll]  \\
        \varepsilon_{11}         \ar@{-}[drr]            & \varepsilon_{12}   \ar@{-}[dr]  &    & \varepsilon_{21}   \ar@{-}[dl]   & \varepsilon_{22}  \ar@{-}[dll]    \\
                     &   & \nabla                  &           &                       \\
}}
\]
\end{tiny}

\subsection{Latin antilattices}
A \emph{Latin square} is an $n\times n$ array filled with $n$ different symbols, where each symbol occurs exactly once in each row and in each column. Two Latin squares are said to be \emph{orthogonal} if, when they are superimposed, the ordered pairs in each entry are all distinct.

Let $N$ be an antilattice of type $(p,q)$ with generating matrix pair $(M,J)$ where the meet matrix $M$ is in normal form. Relabel the elements of $N$ with ordered pairs $(i,j)$ indicating the position of the element in $M$ where $1\leq i\leq p$, $1\leq j\leq q$. In other words, we replace the number $(i-1)q+j$ with the pair $(i,j)$. For $k=1,2$, let $J_k$ be the matrix of $k$th components of the entries of $J$. $N$ is said to be a \emph{Latin antilattice} if $J_1$ and $J_2$ are orthogonal Latin squares.

\begin{theorem}\label{Thm:latin_semimagic}
Every Latin antilattice is semimagic.
\end{theorem}
\begin{proof}
This follows from the method discovered independently by Choi Seok-jeong (1646--1715) and Leonhard Euler (1707--1783) that constructs a semimagic square from any pair of orthogonal Latin squares (\cite{CoDi}, p.~12). Let $N$ be a Latin antilattice of order $n^2$ with meet matrix $M$ in normal form and join matrix $J$. As above, we associate to each entry $(i-1)n+j$ of $J$ the ordered pair $(i,j)$, $1\leq i,j\leq n$. Let $J_1$ and $J_2$ be the matrices of components of the entries of $J$. Since each $J_k$ is Latin, in any row or column of $J$, $i$ and $j$ range over all values from $1$ to $n$. Thus the sum of the entries $(i-1)n + j$ in any row or column of $J$ is found by summing over all $i$ and all $j$. This sum has the same value $n(n^2+1)/2$ and so $J$ is a semimagic square.
\end{proof}

We illustrate the construction as follows.
\[
J_1 =
\begin{array}{|ccc|}
\hline
3 & 1 & 2 \\
1 & 2 & 3 \\
2 & 3 & 1 \\
\hline
\end{array}
\qquad
J_2 =
\begin{array}{|ccc|}
\hline
2 & 1 & 3 \\
3 & 2 & 1 \\
1 & 3 & 2 \\
\hline
\end{array}
\quad\rightarrow\quad
J =
\begin{array}{|ccc|}
\hline
(3,2) & (1,1) & (2,3) \\
(1,3) & (2,2) & (3,1) \\
(2,1) & (3,3) & (1,2) \\
\hline
\end{array}
\]
We replace each pair $(i,j)$ with the number $3(i-1) + j$ to get
\[
J = \begin{array}{|ccc|}
\hline
8& 1& 6 \\
3& 5& 7 \\
4& 9& 2 \\
\hline
\end{array}\,
\]
which is the Lo-Shu semimagic square. Thus the \emph{Lo-Shu antilattice} is determined by the join matrix $J$ and the meet matrix $M$ in normal form.

It is also illuminating to consider a variant of the reverse process. Starting with $J$ as above, subtract $1$ from each entry:
\[
\begin{array}{|ccc|}
\hline
7 & 0 & 5 \\
2 & 4 & 6 \\
3 & 8 & 1 \\
\hline
\end{array}
\]
Then write each entry in base $3$:
\[
\begin{array}{|ccc|}
\hline
21 & 00 & 12 \\
02 & 11 & 20 \\
10 & 22 & 01 \\
\hline
\end{array}
\]
Then detach the two squares:
\[
\begin{array}{|ccc|}
\hline
2 & 0 & 1 \\
0 & 1 & 2 \\
1 & 2 & 0 \\
\hline
\end{array}
\qquad
\begin{array}{|ccc|}
\hline
1 & 0 & 2 \\
2 & 1 & 0 \\
0 & 2 & 1 \\
\hline
\end{array}
\]
This is a pair of orthogonal Latin squares. Adding $1$ to each entry gives us $J_1$ and $J_2$ above.

\smallskip

\begin{proposition}\label{Prp:Euler}
  A Latin antilattice of order $n$ exists if and only if $n = k^2, k > 2$, except for $n = 36$.
\end{proposition}
\begin{proof}
This follows readily from the well-known disproof of Euler's conjecture about the existence of orthogonal pairs of Latin squares \cite{BoSh59,BoShPa60}.
\end{proof}

\begin{proposition}\label{Prp:semi_not_latin}
  There exist semimagic antilattices which are not Latin.
\end{proposition}
\begin{proof}
  The join matrix of the D\"{u}rer antilattice cannot be constructed from a pair of orthogonal Latin squares.
\end{proof}

\section{Elementary and odd antilattices}
\label{Sec:elemodd}
A subantilattice $M$ of an antilattice $N$ is said to be \emph{trivial} if $|M|=1$ and it is said to be \emph{proper} if $M$ is a proper subset of $N$. An antilattice is said to be \emph{elementary} if it has no proper nontrivial subantilattices.

An antilattice is said to be \emph{odd} if it has no subantilattice with $2$ elemens.

\begin{proposition}\label{Prp:elem_odd_simple}
    Every elementary antilattice is odd and simple.
\end{proposition}
\begin{proof}
That an elementary antilattice $N$ is odd just follows from the definitions. If $N$ is not simple, then it is has a nontrivial congruence with a congruence class which is a proper subset of $N$ and not a singleton. By Corollary \ref{Cor:cong_sub}, $N$ is not elementary.
\end{proof}

The converse is false, as we will see later.

The class of elementary antilattices is (trivially) closed under taking subantilattices and homomorphic images, but evidently not direct products, since the direct product of nontrivial simple antilattices is not simple.

\begin{proposition}\label{Prp:odd_quasi}
The class of all odd antilattices is a quasi-variety.
\end{proposition}
\begin{proof}
A $2$-element rectangular band is either a left zero band or a right zero band. Thus there are four $2$-element antilattices.
Therefore an antilattice has no $2$-element subantilattices if and only if the following quasi-identities are satisfied:
\begin{align*}
x\land y = x &\quad \& & y\land x = y &\quad \& & x\lor y = x &\quad \& & y\lor x = y\implies y = x\,,\\
x\land y = x &\quad \& & y\land x = y &\quad \& & x\lor y = y &\quad \& & y\lor x = x\implies y = x\,,\\
x\land y = y &\quad \& & y\land x = x &\quad \& & x\lor y = x &\quad \& & y\lor x = y\implies y = x\,,\\
x\land y = y &\quad \& & y\land x = x &\quad \& & x\lor y = y &\quad \& & y\lor x = x\implies y = x\,.
\end{align*}
These four together with the identities defining antilattices axiomatize odd antilattices, and therefore odd antilattices
form a quasi-variety.
\end{proof}

The question is if these quasi-identities can be replaced by identities, in which case odd antilattices form a variety. Our suspicion is that the answer is no.

\begin{conjecture}
  The quasivariety of odd antilattices is proper, that is, it is not a variety.
\end{conjecture}

\subsection{Characterisations of finite odd antilattices}
\begin{proposition}\label{Prp:odd_square}
Every finite odd antilattice is square.
\end{proposition}
\begin{proof}
Let $N$ be a finite antilattice of type $(p,q,r,s)$ and assume, say, $p > q$, that is, the columns of $M$ are longer than the rows. We may also assume $r \geq s$. Since $pq = rs = n$ it follows that $s < p$. Let $W$ denote the set of elements from the first column of $M$ so that $|W|=p$.
Since $p$ is greater than the number of rows of $J$, it follows from the Pigeonhole Principle, that at least one row of $J$ contains at least two elements of $W$, say $a$ and $b$. Since $a$ and $b$ are collinear in both $M$ and $J$, $\{a,b\}$ is a subantilattice of $N$, so that $N$ is not odd. The other cases are proved similarly.
\end{proof}

\begin{proposition}\label{Prp:odd_no_prime}
  The following are equivalent for an antilattice $N$.
  \begin{enumerate}
    \item $N$ has no subantilattice of nonsquare order;
    \item $N$ has no subantilattice of prime order;
    \item $N$ is odd.
  \end{enumerate}
\end{proposition}
\begin{proof}
  (1)$\implies$(2) is obvious and (2)$\implies$(3) follows from the definition of odd. If $N$ is odd, then any finite subantilattice of $N$, which is also odd (by Proposition \ref{Prp:odd_quasi}), must have square order (by Proposition \ref{Prp:odd_square}). Thus (3)$\implies$(1).
\end{proof}

In the finite case, odd antilattices have a particularly nice characterization.

\begin{theorem}\label{Thm:odd_latin}
A finite antilattice $N$ is odd if and only if it is Latin.
\end{theorem}
\begin{proof}
Let $N$ be an antilattice with generating matrices $M$ and $J$, and assume $M$ is in normal form. Since both Latin antilattices (by definition) and odd antilattices (by Proposition \ref{Prp:odd_square} are square, we may assume that $N$ is square. As before, relabel the elements of $N$ with ordered pairs $(i,j)$ indicating the position of $n(i-1)+j$ in $M$. Assuming $J$ has this same labeling, let $J_1$ and $J_2$ be the matrices of components of the entries of $J$.

A pair $x = (i_x,j_x)$, $y=(i_y,j_y)\in N$ forms a subalgebra of order $2$ if and only if $x$ and $y$ are collinear both in $M$ and $J$. The pair is collinear in $M$ if and only if $i_x = i_y$ or $j_x = j_y$. It is collinear in $J$ if and only if $J_1$ or $J_2$ has a repeated element in one of its rows or columns. Therefore, by finiteness, the nonexistence of a subantilattice of order $2$ is equivalent to the statement that each row and column of $J_1$ and $J_2$ is a permutation, or equivalently, that $J_1$ and $J_2$ are orthogonal Latin squares. This proves the desired equivalence.
\end{proof}

\begin{corollary}\label{Cor:elem_square}
Every elementary antilattice is square. If an elementary antilattice has order $n$, then $n = k^2, k > 2$ and  $n \neq 36$.
\end{corollary}
\begin{proof}
This follows from Theorem \ref{Thm:odd_latin} and Proposition \ref{Prp:Euler}.
\end{proof}

\begin{proposition}
Let $N$ be a finite odd antilattice and let $\alpha$ be a congruence on $N$. Then there are a square number of $\alpha$-congruence classes, all of which have the same size.
\end{proposition}
\begin{proof}
Let $N$ be a finite odd antilattice Then $N$ is square (Proposition \ref{Prp:odd_square}) and each $\alpha$-congruence class, being an odd antilattice (by Corollary \ref{Cor:cong_sub} and Proposition \ref{Prp:odd_quasi}), is also square. Since the classes of $\alpha$ are given by cartesian partitions of the generating matrices of $N$, the desired result follows.
\end{proof}

\begin{corollary}\label{Cor:prime_simple}
Let $p$ be a prime. Any odd antilattice of order $p^2$ is simple.
\end{corollary}

There are numerous odd antilattices of order $p^2$, $p$ prime, that are not elementary. This follows from a result of Heinrich and Zhu \cite{HeZh86} and is described below.

\subsection{Regular antilattices}
An antilattice is \emph{regular} if all four of Green's $\mathcal{L}$- and $\mathcal{R}$-relations (one of each for each operation) are congruences for both operations. Regular antilattices form a variety which was studied in some detail in \cite{CvKiLePi19}. The main decomposition theorem for regular antilattices is as follows.

\begin{proposition}[\cite{CvKiLePi19}, Thm.~3.3]\label{Prp:decomp}
Every nonempty regular antilattice $N$ is isomorphic to a direct product
$N_{\mathcal{L}\mathcal{L}}\times N_{\mathcal{L}\mathcal{R}}\times N_{\mathcal{R}\mathcal{L}}\times N_{\mathcal{R}\mathcal{R}}$ of flat antilattices, with each factor being unique up to isomorphism.
\end{proposition}

Here $N_{\mathcal{L}\mathcal{L}}$ satisfies the identities $x\lor y = x = x\land y$, $N_{\mathcal{L}\mathcal{R}}$ satisfies $x\lor y = x$, $x\land y = y$, and so on. For our purposes, all that matters is that each factor is flat.

\begin{theorem}\label{Thm:reg_odd}
A regular, odd antilattice is trivial.
\end{theorem}
\begin{proof}
Suppose $N = N_{\mathcal{L}\mathcal{L}}\times N_{\mathcal{L}\mathcal{R}}\times N_{\mathcal{R}\mathcal{L}}\times N_{\mathcal{R}\mathcal{R}}$ is regular and nontrivial. Then at least one factor is nontrivial, say $N_{\mathcal{L}\mathcal{L}}$. Choosing distinct $a,b\in N_{\mathcal{L}\mathcal{L}}$, we have that $\{a,b\}$ is a subantilattice. Now fix
$c_{\ell r}\in N_{\mathcal{L}\mathcal{R}}$, $c_{r\ell}\in N_{\mathcal{R}\mathcal{L}}$, and $c_{rr}\in N_{\mathcal{R}\mathcal{R}}$. Then $\{(a,c_{\ell r},c_{r\ell},c_{rr}),(b,c_{\ell r},c_{r\ell},c_{rr})\}$ is a subantilattice of $N$. Therefore $N$ is not odd. The cases where other factors are nontrivial are handled similarly.
\end{proof}

\subsection{Even graphs}
Let $D$ be any double band. The \emph{even graph} $G(D)$ of $D$ is the graph with vertex set $D$ defined as follows:
\smallskip
Two elements $a,b \in D$ are connected by an edge if and only if they $\{a,b\}$ is a subalgebra.
\smallskip

It is not clear which class of graphs is determined by the class of all even graphs of quasilattices or antilattices.
For antilattices, the following observation is immediate.

\begin{proposition}\label{Prp:odd_empty}
  An antilattice is odd if and only if its even graph is empty.
\end{proposition}

\noindent Here ``empty'' is in the graph theoretic sense of having empty edge set.

Even graphs of antilattices may also be defined by their generating matrices.

On an antilattice $N$, define four symmetric relations as follows. For $x,y\in N$,

\noindent $x\sim_{hh} y$ if $x$ and $y$ are in the same row of $M$ and same row of $J$;

\noindent $x\sim_{vv} y$ if $x$ and $y$ are in the same column of $M$ and same column $J$;

\noindent $x\sim_{hv} y$ if $x$ and $y$ are in the same row of $M$ and same column of $J$;

\noindent $x\sim_{vh} y$ if $x$ and $y$ are in the same column of $M$ and same row of $J$.

\noindent Finally, let $x\sim y$ if any of the above is true, that is, $\sim$ is the union of $\sim_{hh}$, $\sim_{vv}$, $\sim_{hv}$ and $\sim_{vh}$.
These symmetric relations define an edge-colored graph $G(N) = (N,\sim)$ where the four colors are: $hh,vv,hv,vh$. This graph is exactly the even graph.

In the case of the D\"{u}rer antilattice, its even graph is a disjoint union of $4$ cycles, spanning the
set $1-16$. Note that in this case only two edge colors are used.

\[
{\xymatrix{
1 \ar@{-}[r] \ar@{-}[d]& 4 \ar@{-}[d] & 2\ar@{-}[r] \ar@{-}[d] & 3\ar@{-}[d] & 5 \ar@{-}[r] \ar@{-}[d]& 8 \ar@{-}[d] & 6\ar@{-}[r] \ar@{-}[d] & 7\ar@{-}[d] \\
13 \ar@{-}[r]& 16 & 14\ar@{-}[r] & 15 & 9 \ar@{-}[r]& 12 & 10\ar@{-}[r] & 11 \\
}}
\]

The Lo-Shu antilattice, on the other hand, is Latin, hence odd, and so its even graph is empty.

\section{Main Result}
\label{Sec:main}
In the previous sections we have introduced several classes of antilattices. The following diagram represents the dependencies among these classes. We prove all implications and equivalences, and give examples showing that the implications are indeed sharp.

\begin{theorem}\label{Thm:main}
The following implications and equivalences hold among classes of finite antilattices.
\[
{\xymatrix{
                         & \textrm{elementary}\ar@{=>}[d]                  &    \\
                         & \textrm{odd and simple}\ar@{=>}[ld]\ar@{=>}[rd] &   \\
\textrm{odd}\ar@{<=>}[r] & \textrm{empty even graph}\ar@{<=>}[d]           & \textrm{simple}\ar@{=>}[d]\\
                         & \textrm{Latin}\ar@{<=>}[lu]\ar@{=>}[d]          & \textrm{irreducible} \\
                         & \textrm{semimagic}\ar@{=>}[d]                   & \\
                         & \textrm{square}
}}
\]
\end{theorem}
We start at the bottom of the diagram. That every semimagic antilattice is square is just by definition. To see the implication is strict, just take any antilattice $N$ with meet matrix $M$ in normal form and $J$ not a semimagic square. No relabeling of $N$ can also have $M$ in normal form and $J$ semimagic.

That every Latin antilattice is semimagic is Theorem \ref{Thm:latin_semimagic}. That the implication is strict is Proposition \ref{Prp:semi_not_latin}.

The equivalence of finite odd antilattices, Latin antilattices, and antilattices with empty even graph are Theorem \ref{Thm:odd_latin} and
Proposition \ref{Prp:odd_empty}.

The strict implication from simple antilattices to irreducible antilattices is covered in Proposition \ref{Prp:simple_irred}.

There is no implication from simple antilattices to odd antilattices because of the existence of nonsquare simple antilattices discussed before.

That every elementary antilattice is odd and simple is Proposition \ref{Prp:elem_odd_simple}.

What remains is to show that there exists an odd antilattice which is not simple, and there exists an odd, simple antilattice which is not elementary. For this we use the powerful results of Heinrich and Zhu \cite{HeZh86} about the existence of orthogonal subsquares in orthogonal Latin squares. Let $LS(v,n)$ denote a pair of orthogonal latin squares of order $v$ such that some $n$ rows and columns define in each square a pair of orthogonal latin squares of order $n$. It is shown that for $v > n > 1$, there exists a $LS(v,n)$ if and only if $v\geq 3n$, $v\neq 6$, $n\neq 2,6$. In the language of antilattices, the existence of orthogonal subsquares is equivalent to the existence of nontrivial subantilattices of an odd antilattice. The main result was obtained by work of several authors in a series of papers and can be stated in terms of antilattices as follows:

\begin{theorem}\label{Thm:odd}
An odd antilattice of order $v^2$ with a subantilattice of order $n^2$ exists if and only if $v\geq 3n$, $v\neq 6$, $n\neq 2,6$.
\end{theorem}

\begin{corollary}
Every odd antilattice of order less than $81$ is elementary.
\end{corollary}

An example of the smallest case of Theorem \ref{Thm:odd}, $n=3$ and $v=9$, gives us the following.

\begin{theorem}\label{Thm:odd_not_simple}
There exist odd antilattices that are not simple.
\end{theorem}
\begin{proof}
We present the generating matrices of an antilattice of order $81$, which is obtained as a product of two elementary antilattices of order $9$. Thus it is odd but neither elementary nor simple. The matrix $9\times 9$ matrix $M$ is in normal form on the entries $\{0,\ldots,80\}$. Here is the matrix
\[
J =
\begin{array}{|rrrrrrrrr|}
\hline       	
0  &  20  &  10  &  60  &  80  &  70  &  30  &  50  &  40    \\
19  &  9  &  2  &  79  &  69  &  62  &  49  &  39  &  32    \\
11  &  1  &  18  &  71  &  61  &  78  &  41  &  31  &  48    \\
57  &  77  &  67  &  27  &  47  &  37  &  6  &  26  &  16    \\
76  &  66  &  59  &  46  &  36  &  29  &  25  &  15  &  8    \\
68  &  58  &  75  &  38  &  28  &  45  &  17  &  7  &  24    \\
33  &  53  &  43  &  3  &  23  &  13  &  54  &  74  &  64    \\
52  &  42  &  35  &  22  &  12  &  5  &  73  &  63  &  56    \\
44  &  34  &  51  &  14  &  4  &  21  &  65  &  55  &  72    \\
\hline
\end{array}
\]
\end{proof}

The following is a corollary of Theorem \ref{Thm:odd} and can be viewed as refining Corollary \ref{Cor:prime_simple}:

\begin{corollary}\label{Cor:prime_simple_nonE}
Let $p$ be a prime. A nonelementary, simple odd antilattice of order $p^2$ exists if and only if $p \geq 11$.
\end{corollary}

For example, we may take $v=11$ and $n=3$ to obtain a nonelementary, simple, odd antilattice of order $121$ containing a subantilattice of order $9$.

Here is a more explicit example, given by a construction from \cite{He77}. It implies the existence of a simple, odd antilattice of order $13^2$ that contains a subantilattice of order $4^2 = 16$. Below is a Latin square $L$ of order $13$ which is \emph{self-orthogonal}, meaning that $L$ and its transpose $L^T$ form a pair of orthogonal Latin squares. The Latin subsquare of order $4$ is depicted in red, and it is also self-orthogonal.
\[
L = \begin{array}{|rrrrrrrrrrrrr|}
\hline
1 & 10 & 11 & 12 & 13 & 8 & 4 & 9 & 5 & 6 & 2 & 7 & 3 \\
6 & 2 & 10 & 11 & 12 & 13 & 9 & 5 & 1 & 7 & 3 & 8 & 4 \\
2 & 7 & 3 & 10 & 11 & 12 & 13 & 1 & 6 & 8 & 4 & 9 & 5 \\
7 & 3 & 8 & 4 & 10 & 11 & 12 & 13 & 2 & 9 & 5 & 1 & 6 \\
3 & 8 & 4 & 9 & 5 & 10 & 11 & 12 & 13 & 1 & 6 & 2 & 7 \\
13 & 4 & 9 & 5 & 1 & 6 & 10 & 11 & 12 & 2 & 7 & 3 & 8 \\
12 & 13 & 5 & 1 & 6 & 2 & 7 & 10 & 11 & 3 & 8 & 4 & 9 \\
11 & 12 & 13 & 6 & 2 & 7 & 3 & 8 & 10 & 4 & 9 & 5 & 1 \\
10 & 11 & 12 & 13 & 7 & 3 & 8 & 4 & 9 & 5 & 1 & 6 & 2 \\
5 & 6 & 7 & 8 & 9 & 1 & 2 & 3 & 4 &\color{red}{10} &\color{red}{12} &\color{red}{13} &\color{red}{11} \\
9 & 1 & 2 & 3 & 4 & 5 & 6 & 7 & 8 &\color{red}{13} &\color{red}{11} &\color{red}{10} &\color{red}{12} \\
4 & 5 & 6 & 7 & 8 & 9 & 1 & 2 & 3 &\color{red}{11} &\color{red}{13} &\color{red}{12} &\color{red}{10} \\
8 & 9 & 1 & 2 & 3 & 4 & 5 & 6 & 7 &\color{red}{12} &\color{red}{10} &\color{red}{11} &\color{red}{13} \\
\hline
\end{array}
\]
The $13\times 13$ generating matrix $M$ of the corresponding antilattice is in normal form in the entries $\{0,\ldots,168\}$ with the generating matrix of the subantilattice being the $4\times 4$ block in the lower right corner:
\[
\begin{array}{|rrrr|}
\hline
\color{red}{126} &\color{red}{127} &\color{red}{128} &\color{red}{129} \\
\color{red}{139} &\color{red}{140} &\color{red}{141} &\color{red}{142} \\
\color{red}{152} &\color{red}{153} &\color{red}{154} &\color{red}{155} \\
\color{red}{165} &\color{red}{166} &\color{red}{167} &\color{red}{168} \\
\hline
\end{array}
\]
Here is the matrix $J$, again with the subantilattice's matrix depicted in red:
\[
J = \begin{array}{|rrrrrrrrrrrrr|}
\hline
0 & 122 & 131 & 149 & 158 & 103 & 50 & 114 & 61 & 69 & 21 & 81 & 33 \\
74 & 14 & 123 & 132 & 150 & 159 & 116 & 63 & 10 & 83 & 26 & 95 & 47 \\
23 & 87 & 28 & 124 & 133 & 151 & 160 & 12 & 76 & 97 & 40 & 109 & 52 \\
89 & 36 & 100 & 42 & 125 & 134 & 143 & 161 & 25 & 111 & 54 & 6 & 66 \\
38 & 102 & 49 & 113 & 56 & 117 & 135 & 144 & 162 & 8 & 68 & 20 & 80 \\
163 & 51 & 115 & 62 & 9 & 70 & 118 & 136 & 145 & 13 & 82 & 34 & 94 \\
146 & 164 & 64 & 11 & 75 & 22 & 84 & 119 & 137 & 27 & 96 & 39 & 108 \\
138 & 147 & 156 & 77 & 24 & 88 & 35 & 98 & 120 & 41 & 110 & 53 & 5 \\
121 & 130 & 148 & 157 & 90 & 37 & 101 & 48 & 112 & 55 & 7 & 67 & 19 \\
57 & 71 & 85 & 99 & 104 & 1 & 15 & 29 & 43 &\color{red}{126} &\color{red}{155} &\color{red}{166} &\color{red}{141} \\
105 & 2 & 16 & 30 & 44 & 58 & 72 & 86 & 91 &\color{red}{167} &\color{red}{140} &\color{red}{129} &\color{red}{152} \\
45 & 59 & 73 & 78 & 92 & 106 & 3 & 17 & 31 &\color{red}{142} &\color{red}{165} &\color{red}{154} &\color{red}{127} \\
93 & 107 & 4 & 18 & 32 & 46 & 60 & 65 & 79 &\color{red}{153} &\color{red}{128} &\color{red}{139} &\color{red}{168} \\
\hline
\end{array}
\]

We conclude with the following.

\begin{problem}
For which orders do finite elementary antilattices exist?
\end{problem}

It would be also be interesting to investigate more general odd noncommutative lattices, such as odd quasilattices.

\section*{Acknowledgements}
This work is supported in part by the Slovenian Research Agency (research program P1-0294 and research projects N1-0032, J1-9187, J1-1690, N1-0140, J1-2481), and in part by H2020 Teaming InnoRenew CoE.


\begin{thebibliography}{99}
\bibliographystyle{line}
\bibliography{JAMS-paper}

\bibitem{Bi35} G. Birkhoff,
On the structure of abstract algebra,
\emph{Proc. Camb. Philos. Soc.} \textbf{31} (1935) 433--354. .

\bibitem{BoSh59} R. C. Bose and S. S. Shrikhande,
On the falsity of Euler's conjecture about the non-existence of two orthogonal Latin squares of order $4t+2$,
\emph{Proc. Nat. Acad. Sci. U.S.A.} \textbf{45} (1959), 734--737.
	
\bibitem{BoShPa60} R. C. Bose, S. S. Shrikhande and E. T. Parker,
Further results on the construction of mutually orthogonal Latin squares and the falsity of Euler's conjecture,
\emph{Canadian J. Math.} \textbf{12} (1960), 189--203.
		
\bibitem{Bu81} S. Burris and H. P. Sankappanavar,
\emph{A Course in Universal Algebra},
Graduate Texts in Mathematics \textbf{78}, Springer-Verlag, New York, 1981.

\bibitem{Cl61} A. H. {Clifford} and G. B. Preston,
\emph{The Algebraic Theory of Semigroups, Volume 1},
Math. Surveys of the American Math. Soc. \textbf{7}, Providence, Rhode Island, 1961.

\bibitem{CoDi} C. J. Colbourn and J. H Dinitz (eds.),
\emph{Handbook of Combinatorial Designs}, 2nd Edition, CRC Press, 2006.

\bibitem{CvKiLePi19} K. Cvetko-Vah, M. Kinyon, J. Leech, and T. Pisanski,
Regular antilattices,
\emph{The Art of Discrete and Applied Mathematics} \textbf{2} (2019), {\#}P2.06.

\bibitem{GePe1989} J. Gerhard and M. Petrich,
Varieties of bands revisited,
\emph{Proc. London Math. Soc.} \textbf{58} (1989), no. 2, 323--350.

\bibitem{How95} J. Howie,
\emph{Fundamentals of Semigroup Theory},
Oxford U. Press, 1995.

\bibitem{He77}  K. Heinrich,
Self-orthogonal Latin squares with self-orthogonal subsquares.
\emph{Ars Combin.} \textbf{3} (1977), 251--266.
		
\bibitem{HeZh86} K. Heinrich and L. Zhu,
Existence of orthogonal Latin squares with aligned subsquares,
\emph{Discrete Math.} \textbf{59} (1986), no. 1-2, 69--78.

\bibitem{LaLe02} G. Laslo and J. Leech,
Green's equivalences on noncommutative lattices,
\emph{Acta Sci. Math. (Szeged)} \textbf{68} (2002), 501-533.

\bibitem{Le05} J. Leech,
Magic squares, finite planes and simple quasilattices,
\emph{Ars Combinatoria} \textbf{77} (2005), 75--96.

\bibitem{Le19} J. Leech,
My journey into noncommutative lattices and their theory,
\emph{The Art of Discrete and Applied Mathematics} \textbf{2} (2019), {\#}P2.01.

\bibitem{Le20} J. Leech,
\emph{Noncommutative Lattices: Skew Lattices, Skew Boolean Algebras and Beyond},
Famnit Lectures Series , University of Primorska Press, 2020.

\bibitem{Mal} A.~I. Mal'cev,
Several remarks on quasivarieties of algebraic systems (Russian),
\emph{Algebra i Logika Sem.} \textbf{5} (1966), no. 3, 3--9.


\bibitem{Sage} SageMath,
the Sage Mathematics Software System (Version 8.7),
The Sage Developers, \url{https://www.sagemath.org}

\end{thebibliography}
\end{document}